\documentclass[12pt,letterpaper]{amsart}
\usepackage{geometry}
\usepackage{graphicx}
\usepackage{amsfonts}
\usepackage{amssymb}
\usepackage{mathrsfs,latexsym,url,graphicx,setspace,color}
\usepackage{mathpazo}
\usepackage{amsmath,amsthm,amscd}
\usepackage{setspace,cancel}
\newtheorem{theorem}{Theorem}
\newtheorem*{theorem*}{Theorem}
\newtheorem{proposition}{Proposition}
\newtheorem{lemma}{Lemma}

\theoremstyle{remark}
\newtheorem{remark}{Remark}

\newtheorem{example}{Example}

\usepackage{mathrsfs}
\usepackage[active]{srcltx}
\usepackage{amssymb,latexsym}

\newcommand{\abs}[1]{\left\lvert#1\right\rvert}

\newcommand{\R}{\mathbb{R}}

\newcommand{\B}{\mathcal{B}}
\newcommand{\disc}{\mathbb{D}}
\newcommand{\C}{\mathbb{C}}
\newcommand{\n}{\mathbb{N}}

\newcommand{\zb}{\overline{z}}
\newcommand{\D}{\Omega}

\providecommand{\MR}{\relax\ifhmode\unskip\space\fi MR }

\providecommand{\href}[2]{#2}

\title[Nilpotent Toeplitz Operators on Reinhardt domains]{Nilpotent Toeplitz Operators on Reinhardt Domains}
\author{Mehmet \c{C}el\.ik}
\address[Mehmet \c{C}elik]{University of North Texas at Dallas, Department of
Mathematics and Information Sciences,  7400 University Hills Blvd., 
Dallas, TX 75241}
\email{mehmet.celik@untdallas.edu}

\author{Yunus E. Zeytuncu}
\address[Yunus E. Zeytuncu]{University of Michigan - Dearborn, Department of Mathematics and Statistics, Dearborn, MI  48128}
\email{zeytuncu@umich.edu}

\date{\today}

\begin{document}
\maketitle

\begin{abstract}
We construct explicit examples of non-trivial nilpotent Toeplitz operators on Bergman spaces of certain Reinhardt domains in $\mathbb{C}^2$.

\end{abstract}

\section{Introduction}

\subsection{Set-up and Result}
Let $\D$ be a domain in $\C^n$ and $A^2(\D)$ denote the Bergman space of $\D$.
The Bergman projection operator, $\B_{\D}$, is the orthogonal projection from $L^2(\D)$ onto $A^{2}(\D)$. It is an integral operator with the kernel called the Bergman kernel, denoted by $B_{\D}(z,w)$. If $\{e_n(z)\}_{n=0}^{\infty}$ is an orthonormal basis for $A^2(\D)$ then the Bergman kernel can be represented as 
$B_{\D}(z,w)=\sum\limits_{n=0}^{\infty}e_n(z)\overline{e_n(w)}$. See \cite{Krantz2001} for general theory of Bergman spaces.

For a function $u$ on $\D$, the Toeplitz operator $T_u: A^2(\D)\rightarrow  A^2(\D)$ with the symbol $u$ is defined by $T_u(f)=\B_{\Omega}(uf)$.

In this note, we are interested in the zero product problem. For two symbols $u_1$ and $u_2$, if the product $T_{u_1} T_{u_2}$ is identically zero on $A^2(\D)$ then can we claim  $T_{u_1}$ or $T_{u_2}$ is identically zero? This is a non-trivial problem and the answer is not even known when $\D$ is the unit disc.

Here, we indicate the problem has a different flavor in higher dimensions. In particular, we present a family of Reinhardt domains in $\C^2$ on which not only zero products of non-trivial Bergman space Toeplitz operators exist but we can find nilpotent Toeplitz operators.
 
\begin{theorem}\label{theorem}
There exists Reinhardt domains in $\C^2$ on whose Bergman spaces there are nilpotent Toeplitz operators.
\end{theorem}

\begin{remark}
It becomes clear in the proof that the operators in Theorem \ref{theorem} are also of infinite rank.
\end{remark}

\subsection{History:} The zero product problem on the Hardy space is initiated in \cite{BrownHalmos63-64}. It is completely solved in \cite{AlemanVukotic2009} where authors established that the product of non-zero Toeplitz operators is never zero. For the intermediate results, before the complete solution, see \cite{Guo1996,Gu2000} and the references in \cite{AlemanVukotic2009}.

In \cite{AhernCuckovic2001-1}, it is shown that for the Toeplitz operators on the Bergman space $A^2(\Bbb D)$ of the unit disc $\Bbb D$, the analogue of the Brown-Halmos theorem holds under some additional hypothesis that $u$ and $v$ are bounded and harmonic. Later, the same result is proven for radial symbols in \cite{AhernCuckovic2001-2}. The problem on $\disc$, without extra assumptions on the symbols, remains open. 

The higher dimensional cases are studied in \cite{ChoeKoo2006,ChoeLeeNamZheng2007,ChoeKooLee2007} where the results on the unit disc are extended to the ball or to the polydisk. In these papers, neither non-trivial zero products nor nilpotent Toeplitz operators are observed.

In \cite{BauerLe2011}, the problem is considered on the Segal-Bargmann space (the space of square integrable entire functions on $\C^n$ with a Gaussian decay weight) and an example of a non-trivial zero product of three Toeplitz operators is constructed. However, no nilpotent Toeplitz operator is observed.

\section{Proof of the Theorem}

Inspired by the construction in \cite{Wiegerinck84}, we define the following family of domains $\D_m$ in $\mathbb{C}^2$. 
\begin{align*}
X&=\left\{(z_1,z_2)\in\C^2\ : |z_1|>e,\ \ |z_2|<\frac{1}{|z_1|\log|z_1|}\right\}\\
Y_{m}&=\left\{(z_1,z_2)\in\C^2\ : |z_2|>2,\ \ \left||z_1|-\frac{1}{|z_2|}\right|<\frac{1}{|z_2|^m}\right\}\\
Z&=\left\{(z_1,z_2)\in\C^2\ : |z_1|\leq e,\ \ |z_2|\leq 2\right\}
\end{align*}
and put 
\[\D_m=X\cup Y_m\cup Z, \ \ \ m=1,2,\ldots.\]

Each $\D_m$ is an unbounded Reinhardt domain with finite volume, see Figure \ref{Fig1}. 
\begin{figure}[ht!]
\centering
\includegraphics[width=80mm]{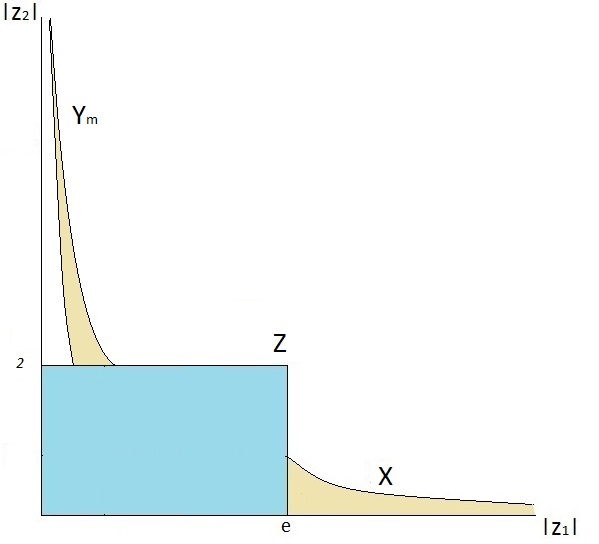}
\caption{Representation of $\Omega_m$ in absolute space $\{(r_1,r_2)\in\R^2\ |\ r_1\geq 0\ \text{ and }\ r_2\geq 0\}$, under the map $\tau:(z_1,z_2)\rightarrow(|z_1|,|z_2|)$.}
\label{Fig1}
\end{figure}

\begin{lemma}\label{the lemma}
For a multi-index $\alpha=(\alpha_1,\alpha_2)$, the monomial $z^{\alpha}$ is in $A^2(\D_m)$ 
 if and only if $\alpha_2\geq\alpha_1>\alpha_2-\frac{m-1}{2}$.
\end{lemma}

\begin{proof}
We start with the calculation on the domain $X$.
\begin{align*}
\int\limits_{X}\abs{z^{\alpha}}^2dV(z)
&=\int\limits_{\abs{z_1}>e}\abs{z_1}^{2\alpha_1}\int\limits_{|z_2|<\frac{1}{|z_1|\log|z_1|}}\abs{z_{2}}^{2\alpha_2}dA(z_{2})dA(z_{1})\\
&=4\pi^2\int\limits_{e}^{\infty} r_1^{2\alpha_1+1}\int\limits_{0}^{\frac{1}{r_1 \log(r_1)}} r_{2}^{2\alpha_2+1}dr_2 dr_1
=\frac{4\pi^{2}}{2\alpha_2+2}\int\limits_{e}^{\infty} \frac{r_1^{2\alpha_1+1}}{r_1^{2\alpha_2+2}(\log(r_1))^{2\alpha_2+2}} dr_1.
\end{align*}
We note that for $k>0$, the improper integral $\int_{e}^{\infty}\frac{1}{x^m(\log x)^k}dx$ converges if and only if $m\geq 1$.
Therefore, the last integral above (where $k=2\alpha_2+2>0$ and $m=2(\alpha_2-\alpha_1)+1$) is finite if and only if $(\alpha_2- \alpha_1)\geq 0$. In other words
\begin{align}\label{condition1}
z^{\alpha}\in A^2(X) \ \ \Leftrightarrow\ \ \alpha_2\geq \alpha_1.
\end{align}

We continue with the calculation on the domain $Y_m$.

\begin{align*}
\int\limits_{Y_m}\abs{z^{\alpha}}^{2}dV(z)
&=\int\limits_{\abs{z_2}>2}\abs{z_2}^{2\alpha_2}\int\limits_{\frac{1}{|z_2|}-\frac{1}{|z_2|^m}<|z_1|<\frac{1}{|z_2|}+\frac{1}{|z_2|^m}}\abs{z_1}^{2\alpha_1}dA(z_{2})dA(z_{1})\\
&=4\pi^2\int\limits_{2}^{\infty} r_2^{2\alpha_2+1}\int\limits_{\frac{1}{r_2}-\frac{1}{r_2^m}}^{\frac{1}{r_2}+\frac{1}{r_2^m}} r_{1}^{2\alpha_1+1}dr_1 dr_2\\
&=\frac{4\pi^2}{2\alpha_1+2}\int\limits_{2}^{\infty} r_2^{2\alpha_2+1}\left[\left(\frac{1}{r_2}-\frac{1}{r_2^m}\right)^{2\alpha_1+2}-\left(\frac{1}{r_2}+\frac{1}{r_2^m}\right)^{2\alpha_1+2}\right]dr_2.
\end{align*}
Since $r_2>2$, after using the binomial expansion in the brackets we consider the term $1/r_2$ with the smallest degree as the dominant, which is $1/r_2^{2\alpha_1+1+m}$.  The last integral can be estimated by
\begin{align*}
\int\limits_{Y_m}\abs{z^{\alpha}}^{2}dV(z) \approx \int\limits_{2}^{\infty} r_2^{2\alpha_2+1} \frac{1}{r_2^{2\alpha_1+1+m}}dr_2.
\end{align*}
The integral on the right is finite if and only if $\alpha_1 > \alpha_2+\frac{1-m}{2}$. In other words
\begin{align}\label{condition2}
z^{\alpha}\in A^2(Y_m) \ \ \Leftrightarrow\ \ \alpha_1 > \alpha_2+\frac{1-m}{2}.
\end{align}
The lemma follows from \eqref{condition1} and \eqref{condition2}.

\end{proof}

Next, we set $m\geq 6$, $\phi=z_1/\zb_1$ and consider $T_{\phi}$ on $A^2(\D_m)$.
\begin{proposition}\label{proposition}
The following properties hold.
\begin{itemize}
\item[$(i)$] $T_{\phi}$ is not a zero operator.
\item[$(ii)$] $T_{\phi}$ does not have finite rank.
\item[$(iii)$] $T_{\phi}$ is a bounded operator.
\item[$(iv)$] $T_{\phi}$ is a nilpotent operator of degree $\lfloor \frac{m}{4}\rfloor$, the largest integer less than or equal to $\frac{m}{4}$.
\end{itemize}
\end{proposition}

\begin{remark}
Once we prove Proposition \ref{proposition}, we immediately obtain Theorem \ref{theorem}. However, it will be clear in the proof that the domain and the operator we present aren't unique but part of a family of domains and operators. We leave exploration of more examples to the reader.
\end{remark}

Before we start proving Proposition \ref{proposition}, we define the following lattice for $m\geq 6$.
\begin{align*}
R_m&=\left\{(\alpha_1,\alpha_2)\in \n^2\ |\ \alpha_2\geq\alpha_1>\alpha_2-\frac{m-1}{2}\right\}\\
&=\left\{(\alpha_1,\alpha_2)\in \n^2\ |\ \alpha_1+\frac{m-1}{2}>\alpha_2\geq\alpha_1\right\}.
\end{align*}

\begin{remark}\label{remark on Rm}
Shifting $\alpha_1$ to the right by a number $s$ greater than or equal to $\frac{m-1}{2}$ is enough to put the resulting index $(\alpha_1+s,\alpha_2)$ out of $R_m$.
That is, if $(\alpha_1,\alpha_2)\in R_m$ then for $s\geq \frac{m-1}{2}$ we get $\left(\alpha_1+s,\alpha_2\right)\not \in R_m$.
\end{remark}

For a multi-index $\gamma=(\gamma_1,\ldots,\gamma_n)\in \n^{n}$, we set 
$c_{\gamma}^{2}=\int\limits_{\D}\abs{z^{\gamma}}^{2}dV(z).$
Then, on a radially symmetric domain $\D$ that contains the origin, the set (or a subset of) $\left\{\frac{z^{\gamma}}{c_{\gamma}}\right\}_{\gamma\in \n^{n}}$ gives a complete orthonormal basis for $A^{2}(\D)$. Each $f\in A^{2}(\D)$ can be written in the form $f(z)=\sum\limits_{\gamma}f_{\gamma}\frac{z^{\gamma}}{c_{\gamma}}$ where the sum converges in $A^{2}(\D)$, but also uniformly on compact subset of $\D$. For the coefficients $f_{\gamma}$, we have $f_{\gamma}= \left\langle f(z),\frac{z^{\gamma}}{c_{\gamma}}\right\rangle_{\D}$.

\begin{proof}[Proof of Proposition \ref{proposition}]
Consider $T_{\phi}$ on $A^2(\D_m)$ for $m\geq 6$. 
$\D_m$ is a radially symmetric domain and the monomials with exponents that reside in $R_m$  form a complete system for $A^2(\D_m
)$. By using the orthogonality of monomials we get
\begin{align}\label{survival eq}
T_{\phi}(z^{\alpha})=\B_{\D_m}\left(\frac{z_1}{\zb_1}\cdot z^{\alpha}\right)&=
\sum\limits_{\gamma\in R_m}\left\langle \frac{z_1}{\zb_1}z^{\alpha},\frac{z^{\gamma}}{c_{\gamma}}\right\rangle \frac{z^{\gamma}}{c_{\gamma}}\\
\nonumber&=\frac{c_{(\alpha_1+1,\alpha_2)}^2}{c_{(\alpha_1+2,\alpha_2)}^2}z_1^{\alpha_1+2}z_2^{\alpha_2}.
\end{align}
On the above summation only $(\gamma_1,\gamma_2)=(\alpha_1+2,\alpha_2)$ survives. Moreover, there exists multi-indices $(\alpha_1,\alpha_2)$ in $R_m$ such that $(\alpha_1+2,\alpha_2)$ is also in $R_m$. Therefore, there exists $z^{\alpha}\in A^2(\D_m)$ such that $T_{\phi}(z^{\alpha})=\frac{c_{(\alpha_1+1,\alpha_2)}^2}{c_{(\alpha_1+2,\alpha_2)}^2}z_1^{\alpha_1+2}z_2^{\alpha_2}\in A^2(\D_m)$  and $T_{\phi}$ is a non-zero operator. 


For $m\geq 6$ and $k\in\mathbb{N}$, $z_1^{k}z_2^{k+2}\in A^2(\D_m)$ and $T_{\phi}\left(z_1^{k}z_2^{k+2}\right)=\frac{c_{(k+1,k+2)}^2}{c_{(k+2,k+2)}^2}z_1^{k+2}z_2^{k+2}\in A^2(\D_m)$.
Hence, the range of the operator $T_{\phi}$ contains all the monomials of the form $z_1^{k+2}z_2^{k+2}$
and so the range of $T_{\phi}$ is infinite dimensional.\medskip

If $g(z_1,z_2)\in A^2(\D_m)$ then its series expansion will be 
\[g(z_1,z_2)=\sum\limits_{\alpha_1=0}^{\infty}\sum\limits_{\alpha_2=\alpha_1}^{\alpha_1+r-1}g_{\alpha_1\alpha_2}\frac{z_2^{\alpha_2}z_1^{\alpha_1}}{c_{(\alpha_1,\alpha_2)}}=\sum\limits_{\alpha_1=0}^{\infty}\sum\limits_{\alpha_2=\alpha_1}^{\alpha_1+r-1} \left\langle g(z),\frac{z^{\alpha}}{c_{\alpha}}\right\rangle \frac{z^{\alpha}}{c_{\alpha}},\] 
where 
\[r =
\begin{cases}
\frac{m}{2}, & \text{if }m\text{ is even}\\
\frac{m-1}{2}, & \text{if }m\text{ is odd}
\end{cases}.\]
The norm of $g(z_1,z_2)$
is given by
\begin{align}\label{norm of g}
\left\|g\right\|_{A^2(\D_m)}^2
=\sum\limits_{\alpha_1=0}^{\infty}\sum\limits_{\alpha_2=\alpha_1}^{\alpha_1+r-1} \left|g_{\alpha_1\alpha_2}\right|^2
\end{align}
and the norm of $T_{\phi}(g)$ is
\begin{align}\label{norm of Tg}
\left\|T_{\phi}(g)\right\|_{A^2(\D_m)}^2&=\left\|\sum\limits_{\alpha_1=0}^{\infty}\sum\limits_{\alpha_2=\alpha_1}^{\alpha_1+r-1} \left\langle\frac{z_1}{\zb_1}\cdot g(z),\frac{z^{\alpha}}{c_{\alpha}}\right\rangle \frac{z^{\alpha}}{c_{\alpha}}\right\|^2\\
\nonumber&=\sum\limits_{\alpha_1=0}^{\infty}\sum\limits_{\alpha_2=\alpha_1}^{\alpha_1+r-1} \left|\left\langle \frac{z_1}{\zb_1}\cdot g(z),\frac{z^{\alpha}}{c_{\alpha}}\right\rangle\right|^2=\sum\limits_{\alpha_1=0}^{\infty}\sum\limits_{\alpha_2=\alpha_1}^{\alpha_1+r-1} \left|\left\langle \frac{z_1}{\zb_1}\cdot \sum_{\beta}g_{\beta}\frac{z^{\beta}}{c_{\beta}},\frac{z^{\alpha}}{c_{\alpha}}\right\rangle\right|^2\\
\nonumber&=\sum\limits_{\alpha_1=2}^{\infty}\sum\limits_{\alpha_2=\alpha_1}^{\alpha_1+r-1} \left|\left\langle \frac{z_1}{\zb_1}\cdot g_{(\alpha_1-2,\alpha_2)}z_1^{\alpha_1-2}z_2^{\alpha_2},\frac{z^{\alpha}}{c_{\alpha}}\right\rangle\right|^2 ~\text{ by orthogonality of monomials}\\
\nonumber&=\sum\limits_{\alpha_1=2}^{\infty}\sum\limits_{\alpha_2=\alpha_1}^{\alpha_1+r-1} \left|g_{(\alpha_1-2,\alpha_2)}\frac{c_{(\alpha_1-1,\alpha_2)}}{c_{\alpha}}
\right|^2~\text{ then we shift the indices}\\
\nonumber&=\sum\limits_{\alpha_1=0}^{\infty}\sum\limits_{\alpha_2=\alpha_1}^{\alpha_1+r-1} \left|\widetilde{g}_{\alpha_1\alpha_2}\right|^2,
\end{align}
where 
\begin{align*}
 \widetilde{g}_{\alpha_1\alpha_2}=
\begin{cases}
0, & \text{if }\alpha_1=\alpha_2 \text{ or }\alpha_1=\alpha_2+1\\
\frac{c_{(\alpha_1+1,\alpha_2)}}{c_{(\alpha_1+2,\alpha_2)}}g_{\alpha_1\alpha_2}, & \text{ otherwise }
\end{cases}.
\end{align*}
The ratio $\frac{c_{(\alpha_1+1,\alpha_2)}^2}{c_{(\alpha_1+2,\alpha_2)}^2}$ is uniformly bounded by a constant. 
Indeed, each integral on $X$ and $Y_m$ has a uniform bound from above (say $C_X$ and $C_{Y_m}$) because of the conditions \eqref{condition1} and \eqref{condition2}. 
Furthermore, we compute the integrals on the polydisc $Z$ explicitly and estimate as follows
\begin{align}
\frac{c_{(\alpha_1+1,\alpha_2)}^2}{c_{(\alpha_1+2,\alpha_2)}^2}
\leq\frac{C_X+C_{Y_m}
+\pi\frac{e^{2\alpha_1+4}}{\alpha_1+2} \cdot \pi\frac{2^{2\alpha_2+2}}{\alpha_2+1}}
{\pi\frac{e^{2\alpha_1+2}}{\alpha_1+1} \cdot \pi\frac{2^{2\alpha_2+2}}{\alpha_2+1}}
\leq \frac{C_X+C_{Y_m}}{\pi^2}+e^2=C.
\end{align}
This estimate implies 
\begin{align}\label{norm of Tg_2}
\left|\widetilde{g}_{\alpha_1\alpha_2}\right|^2\leq C\cdot\left|g_{\alpha_1\alpha_2}\right|^2,\ \ \text{ for all }\ \  (\alpha_1,\alpha_2)\in R_m.
\end{align}
Thus, from \eqref{norm of g}, \eqref{norm of Tg}, and \eqref{norm of Tg_2} it follows that 
\[\left\|T_{\phi}(g)\right\|_{A^2(\D_m)}^2\leq C\cdot \left\|g\right\|_{A^2(\D_m)}^2.\]

Finally, we calculate the powers of $T_{\phi}$.
\begin{align}\label{survival eq 2}
T_{\phi}^2(z^{\alpha})&=T_{\phi}\cdot T_{\phi}(z^{\alpha})
=T_{\phi}\left(\frac{c_{(\alpha_1+1,\alpha_2)}^2}{c_{(\alpha_1+2,\alpha_2)}^2}z_1^{\alpha_1+2}z_2^{\alpha_2}\right)\\
\nonumber &=\frac{c_{(\alpha_1+1,\alpha_2)}^2}{c_{(\alpha_1+2,\alpha_2)}^2} \cdot \frac{c_{(\alpha_1+3,\alpha_2)}^2}{c_{(\alpha_1+4,\alpha_2)}^2} z_1^{\alpha_1+4}z_2^{\alpha_2}.
\end{align}
As for the third power, 
\begin{align}\label{survival eq 3}
T_{\phi}^3(z^{\alpha})=\frac{c_{(\alpha_1+1,\alpha_2)}^2}{c_{(\alpha_1+2,\alpha_2)}^2} \cdot \frac{c_{(\alpha_1+3,\alpha_2)}^2}{c_{(\alpha_1+4,\alpha_2)}^2}\cdot \frac{c_{(\alpha_1+5,\alpha_2)}^2}{c_{(\alpha_1+6,\alpha_2)}^2} z_1^{\alpha_1+6}z_2^{\alpha_2}.
\end{align}
Continuing in that fashion, the $k^{\text{th}}$ power of the operator is
\begin{align}\label{survival eq 4}
T_{\phi}^k(z^{\alpha})=\frac{c_{(\alpha_1+1,\alpha_2)}^2}{c_{(\alpha_1+2,\alpha_2)}^2} \cdot \frac{c_{(\alpha_1+3,\alpha_2)}^2}{c_{(\alpha_1+4,\alpha_2)}^2} \cdots \frac{c_{(\alpha_1+2k-1,\alpha_2)}^2}{c_{(\alpha_1+2k,\alpha_2)}^2} z_1^{\alpha_1+2k}z_2^{\alpha_2}.
\end{align}\\
\\
In \eqref{survival eq 4}, if $2k<r$ then there exists $(\alpha_1,\alpha_2)\in R_m$ such that $(\alpha_1+2k,\alpha_2) \in R_m$, see the discussion in Remark \ref{remark on Rm}, so $z_1^{\alpha_1+2k}z_2^{\alpha_2}\in A^2(\D_m)$ and $T_{\phi}^k\not\equiv 0$ on $A^2(\D_m)$.\\
\\
However, in \eqref{survival eq 4}, if $2k\geq r$ then for all $(\alpha_1,\alpha_2)\in R_m$ we have $(\alpha_1+2k,\alpha_2)\not \in R_m$ by Remark \ref{remark on Rm}, so we see that $z_1^{\alpha_1+2k}z_2^{\alpha_2}\not\in A^2(\D_m)$ and $T_{\phi}^k\equiv 0$ on $A^2(\D_m)$. That is, $T_{\phi}$ is a nilpotent operator of degree $k$ on $A^2(\D_m)$.

\end{proof}


We illustrate the main arguments of the proof in the following example. 

\begin{example}\label{example} 
Set $m=9$, then the monomial $z_1^{\alpha_1}z_2^{\alpha_2}$ is in $A^2(\D_9)$ if and only if $\alpha_1+4>\alpha_2\geq\alpha_1$. The exponents of the monomial in $A^2(\D_9)$ 
are marked on the lattice below in Figure \ref{Fig2}.

\begin{figure}[ht!]
\centering
\includegraphics[width=70mm]{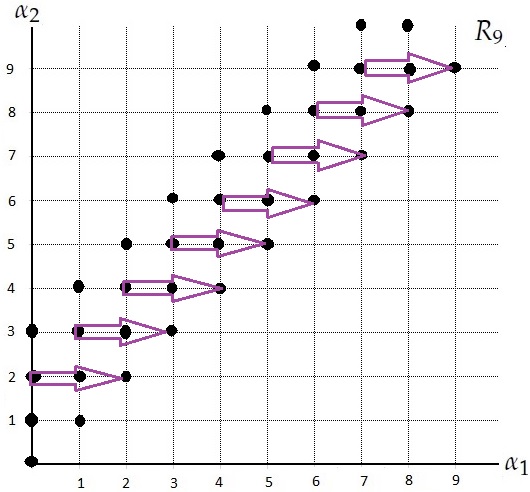}
\caption{Representation of the lattice $R_m$ for $m=9$ and the action of $T_{\phi}$ on $R_m$.}
\label{Fig2}
\end{figure}
It can be noted that $T_{\phi}$ acts like a shift on the lattice, it takes $(\alpha_1,\alpha_2+2)$ to $(\alpha_1+2,\alpha_2+2)$. Thus,  if $T_{\phi}$ is applied on any monomial two times then the exponent of the monomial runs out of the lattice $R_9$. That is, if $z_1^{\alpha_1}z_2^{\alpha_2}\in A^2(\D_9)$ then $T_{\phi}\cdot T_{\phi}(z_1^{\alpha_1}z_2^{\alpha_2})=\frac{c_{(\alpha_1+1,\alpha_2)}^2}{c_{(\alpha_1+2,\alpha_2)}^2} \cdot \frac{c_{(\alpha_1+3,\alpha_2)}^2}{c_{(\alpha_1+4,\alpha_2)}^2} z_1^{\alpha_1+4}z_2^{\alpha_2}\not\in A^2(\D_9)$ and so $T_{\phi}^2\equiv 0$ on $A^2(\D_9)$. 
\end{example}

\section{Acknowledgment} 

We thank the referees for constructive comments on Proposition \ref{proposition} and also for the useful editorial remarks on the exposition of the article.

\bibliographystyle{amsalpha}
\bibliography{CelikZeytuncuBib}
\end{document}